\documentclass[secthm,seceqn,amsthm,ussrhead,10pt]{amsart}
\usepackage{amsmath,latexsym}
\usepackage[english]{babel}
\usepackage[psamsfonts]{amssymb}
\usepackage{times}
\usepackage{cite}
\usepackage{pdflscape} 
\usepackage{ulem}
\usepackage[mathcal]{euscript}
\usepackage{tikz}
\usepackage{hyperref}
\usepackage{cancel}
\usepackage{stmaryrd}
\usetikzlibrary{arrows}

{\theoremstyle{plain}%
  \newtheorem{theorem}{Theorem}
  \newtheorem{corollary}[theorem]{Corollary}
  \newtheorem{proposition}[theorem]{Proposition}
  \newtheorem{lemma}[theorem]{Lemma}
  \newtheorem{example}[theorem]{Example}
  
  \newtheorem{definition}[theorem]{Definition}
  \newtheorem{remark}[theorem]{Remark}}

\newfont{\hueca}{msbm10}

\usepackage{stmaryrd}
\usepackage{xcolor}

\begin{document}

\title{$k$-Modules over linear spaces by $n$-linear maps admitting a multiplicative basis}

\thanks{The first and third authors acknowledge financial assistance by the Centre for Mathematics of the University of Coimbra -- UID/MAT/00324/2013, funded by the Portuguese Government through FCT/MEC and co-funded by the European Regional Development Fund through the Partnership Agreement PT2020. The second author is supported by RFBR 17-01-00258 and he also acknowledges the financial support by the Centre for Mathematics of the University of Coimbra to visit the University of Coimbra. Third author is also supported by the PCI of the UCA `Teor\'\i a de Lie y Teor\'\i a de Espacios de Banach' and, by the PAI with project number FQM298 and he acknowledges the Funda\c{c}\~{a}o para a Ci\^{e}ncia e a Tecnologia for the grant with reference SFRH/BPD/101675/2014.}

\author[E. Barreiro]{E. Barreiro}
\address{Elisabete Barreiro. CMUC, Universidade de Coimbra. Coimbra, Portugal.}
\email{mefb@mat.uc.pt}

\author[I. Kaygorodov]{I. Kaygorodov}
\address{Ivan Kaygorodov. CMCC, Universidade Federal do ABC. Santo Andr\'e, Brasil.}
\email{kaygorodov.ivan@gmail.com}

\author[J.M. S\'anchez]{J.M. S\'anchez}
\address{Jos\'e M. S\'anchez. CMUC, Universidade de Coimbra. Coimbra, Portugal.}
\email{txema.sanchez@mat.uc.pt}

\begin{abstract}
We study the structure of certain $k$-modules $\mathbb{V}$ over linear spaces $\mathbb{W}$ with restrictions neither on the dimensions of $\mathbb{V}$ and $\mathbb{W}$ nor on the base field $\mathbb F$. A basis $\mathfrak B = \{v_i\}_{i\in I}$ of $\mathbb{V}$ is called multiplicative with respect to the basis $\mathfrak B' = \{w_j\}_{j \in J}$ of $\mathbb{W}$ if for any $\sigma \in S_n,$ $i_1,\dots,i_k \in I$ and $j_{k+1},\dots, j_n \in J$ we have $[v_{i_1},\dots, v_{i_k}, w_{j_{k+1}}, \dots, w_{j_n}]_{\sigma} \in \mathbb{F}v_{r_{\sigma}}$ for some $r_{\sigma} \in I$. We show that if $\mathbb{V}$ admits a multiplicative basis then it decomposes as the direct sum $\mathbb{V} = \bigoplus_{\alpha} V_{\alpha}$ of well described $k$-submodules $V_{\alpha}$ each one admitting a multiplicative basis. Also the minimality of $\mathbb{V}$ is characterized in terms of the multiplicative basis and it is shown that the above direct sum is by means of the family of its minimal $k$-submodules, admitting each one a multiplicative basis. Finally we study an application of $k$-modules with a multiplicative basis over an arbitrary $n$-ary algebra with multiplicative basis.



\medskip
\end{abstract}

\maketitle

%%%%%%%%%%%%%%%%%%%%%%%%%%%%%%%%%%%%%%%%%%%%%%%%%%%%%%%%%%%%%
%%%%%%%%%%%%%%%%%%%%%%%%%%%%%%%%%%%%%%%%%%%%%%%%%%%%%%%%%%%%%
\section{Introduction and previous definitions}
%%%%%%%%%%%%%%%%%%%%%%%%%%%%%%%%%%%%%%%%%%%%%%%%%%%%%%%%%%%%%
%%%%%%%%%%%%%%%%%%%%%%%%%%%%%%%%%%%%%%%%%%%%%%%%%%%%%%%%%%%%%

In the literature it is usual to describe an algebra by exhibiting a multiplicative table among the elements of a fixed basis. There exist many classical examples of algebras admitting multiplicative bases in the setting of several algebras as associative algebras, Lie algebras, Malcev algebras, Leibniz algebras, hom-Lie algebras, etc. For instance, in the class of associative algebras we have that the classes of full matrix algebras, group-algebras and quiver algebras when $\mathbb{F}$ is algebraically closed are examples of (associative) algebras admitting multiplicative basis \cite{4,m2,6,17,23}. In the class of Lie algebras we can consider the semisimple finite-dimensional Lie algebras over algebraically closed fields of characteristic 0, semisimple separable $L^*$-algebras \cite{25}, semisimple locally finite split Lie algebras \cite{26}, Heisenberg algebras \cite{22}, twisted Heisenberg algebras \cite{1} or the split Lie algebras considered in \cite[Section 3]{10}. In the class of Leibniz algebras we have the classes of (complex) finite-dimensional naturally graded filiform Leibniz algebras and $n$-dimensional filiform graded filiform Leibniz algebras of length $n-1$ (see \cite{3}). By looking at the multiplication table of the non-Lie Malcev algebra $A_0$ (7-dimensional algebra over its centroid), in \cite[Section 6]{24} we have another example of algebra with a multiplicative basis. In \cite{19} we can find examples of hom-Jordan algebras admitting multiplicative bases. For Zinbiel algebras we have, for instance, that any complex $n$-dimensional nul-filiform Zinbiel algebra admits a multiplicative basis (see \cite{2}). We can also mention the simple Filippov algebras in \cite{Filippov} as examples of $n$-ary algebras admitting a multiplicative basis ($n=3,4,\dots$).

In 2004, Dzhumadil'daev studied $n$-Lie modules over $n$-Lie algebras $\mathbb{A}$ as a vector space $\mathbb{V}$ such that a semi-direct sum $\mathbb{A}+\mathbb{V}$ is once again a $n$-Lie algebra (see \cite{Dzhumadil'daev}). So a module of $n$-Lie algebra is an usual module of Lie algebra if $n=2$. The increasing interest in the study of modules over different classes of algebras, and so over linear spaces, is specially motivated by their relation with mathematical physics (see \cite{phy1, phy2, phy3, Vinogradov, phy4, phy5, chino1, chino2}).

The present paper is devoted to the study of arbitrary $k$-modules over arbitrary linear spaces admitting a multiplicative basis, by focussing on its structure. Calder\'on and Navarro introduced in \cite{Yo_Arb_Alg} the concept of arbitrary algebras admitting a multiplicative basis, and later in \cite{Yo_modules} it was extended to the setup of arbitrary modules over arbitrary linear spaces admitting a multiplicative basis.

The paper is organized as follows. In Section 2, and by inspiring in the connections in the index set techniques developed for arbitrary algebras in \cite{Yo_Arb_Alg}, we consider connections techniques to our set of indexes $I$ of the multiplicative basis so as to get a powerful tool for the study of this class of $k$-modules. By making use of these techniques we show that any $k$-module over a linear space admitting a multiplicative basis is of the form $\mathbb{V} = \oplus_{\alpha} V_{\alpha}$ with any $V_{\alpha}$ a well described $k$-submodule of $\mathbb{V}$ admitting also a multiplicative basis. In Section 3 the minimality of $\mathbb{V}$ is characterized in terms of the multiplicative basis and it is shown that, in case the basis is $\mu$-multiplicative, the above decomposition of $\mathbb{V}$ is actually by means of the family of its minimal simple $k$-submodules. In the final section we study an application of $k$-modules considering $n$-ary algebras. Throughout this paper $\mathbb{V}$ denotes an arbitrary $k$-module over an arbitrary linear space $\mathbb{W}$ in the sense that there are not restrictions on the dimensions of $\mathbb{V}$ and $\mathbb{W}$ or on the base field $\mathbb{F}$ (same for both algebraic structures).

We denote by $S_n$ the permutation group of $n$ elements. For any $\sigma \in S_n,$ we use the notation $$[w_1,\dots,w_j,\dots,w_n]_{\sigma} = [\underbrace{w_1}_{{\sigma(1)-{\rm pos}}},\dots,\underbrace{w_j}_{{\sigma(j)-{\rm pos}}},\dots,\underbrace{w_n}_{{\sigma(n)-{\rm pos}}}],$$ to mean that the element $w_j$ is placed in position $\sigma(j)$ in the $n$-linear map.

\begin{definition}\rm
Let $n \in \mathbb{N}$ and fix $k \in \mathbb{N}$ such that $1 \leq k \leq n$. Let $\mathbb{V}$ be a vector space over an arbitrary base field $\mathbb{F}$. It is said that $\mathbb{V}$ is $k${\it -moduled by a linear space} $\mathbb{W}$ (over same base field $\mathbb{F}$), or just that $\mathbb{V}$ is a $k${\it -module} over $\mathbb{W}$ if it is endowed with a $n$-linear map $$[\mathbb{V},\overset{k)}{\dots},\mathbb{V},\mathbb{W},\overset{n-k)}{\dots},\mathbb{W}]_{\sigma} \subset \mathbb{V},$$ for any $\sigma \in S_n$.
\end{definition}

\begin{example}\rm
Trivially common modules over linear spaces are examples of $k$-modules with $n=2$ and $k=1$. So the present paper contains results that generalize the results from \cite{Yo_modules}.
\end{example}

\begin{example}\rm
We recall that an $n$-ary algebra $\mathbb{A}$ is just a linear space over $\mathbb{F}$ endowed with a $n$-linear map $\langle \cdot, \dots, \cdot \rangle : \mathbb{A} \times \overset{n)}{\cdots} \times \mathbb{A} \to \mathbb{A}$ called the $n${\it -product} of $\mathbb{A}$. By depending on the identities satisfied by the $n$-product we can speak about $n$-ary associative, $n$-Lie algebras, etc. Clearly, any kind of $n$-ary algebra $\mathbb{A}$ is a $k$-module over itself in the case $n=k$. Hence, the present work extends \cite{Yo_n_algebras}.
\end{example}

\begin{example}\rm
Let $S$ be a color $n$-ary algebra, graded by a group $G$ (see \cite{KP2016}). That is, $(S, [\cdot,\dots,\cdot])$ is an algebra that decomposes as the direct sum $S = \oplus_{g \in G}S_g$ in such a way that $[S_{g_1},\dots,S_{g_n}] \subset S_{g_1+\cdots+g_n}.$ Then $S$ is a $k$-module over the linear space $S_g$ under the natural action.
\end{example}

\begin{definition}\label{11}\rm
Let $\mathbb{V}$ be a $k$-module over a linear space $\mathbb{W}$. Given a basis $\mathfrak{B}'=\{w_j\}_{j\in J}$ of $\mathbb{W}$ we say that a basis $\mathfrak{B}=\{v_i\}_{i\in I}$ of $\mathbb{V}$ is {\it multiplicative} with respect to $\mathfrak{B}'$ if for any $\sigma \in S_n,$ $i_1,\dots,i_k \in I$ and $j_{k+1},\dots, j_n \in J$ we have $[v_{i_1},\dots, v_{i_k},w_{j_{k+1}},\dots,w_{j_n}]_{\sigma} \in \mathbb{F}v_{r_{\sigma}}$ for some (unique) $r_{\sigma} \in I$.
\end{definition}

\begin{remark}\rm
Let us observe that Definition \ref{11} agrees with the case for arbitrary algebras studied in \cite{Yo_Arb_Alg} (indeed it is the particular case for $n=2$) and it is a little bit more general than the usual one in the literature (cf. \cite{m5,m2,m4,Ref,m3}) because in these references it is not supposed uniqueness on the element $j \in I$.
\end{remark}

To get examples of $k$-modules admitting multiplicative basis we have just to consider two $\mathbb{F}$-linear spaces $\mathbb{V}$ and $\mathbb{W}$ with their respective basis $\{v_i\}_{i \in I}$ and $\{w_j\}_{j \in J}.$ Let us fix two arbitrary mappings $$\alpha_{\sigma} : I \times \overset{k)}{\cdots} \times I \times J \times \overset{n-k)}{\cdots} \times J \to I \hspace{0.3cm} \mbox{and} \hspace{0.3cm} \beta_{\sigma} : I \times \overset{k)}{\cdots} \times I \times J \times \overset{n-k)}{\cdots} \times J \to \mathbb{F},$$ where as before this notation means that the element $x_l$ is placed in position $\sigma(l)$ in the $n$-linear maps. Then the space $\mathbb{V}$ with the $n$-linear map defined by $$[v_{i_1},\dots,v_{i_k}, w_{j_{k+1}},\dots,w_{j_n}]_{\sigma} := \beta_{\sigma}(i_1,\dots,i_k,j_{k+1},\dots,j_n)v_{\alpha_{\sigma}(i_1,\dots,i_k,j_{k+1},\dots,j_n)}$$ becomes a $k$-module admitting ${\mathcal B}$ as multiplicative basis.

%%%%%%%%%%%%%%%%%%%%%%%%%%%%%%%%%%%%%%%%%%%%%%%%%%%%%%%%%%%%%
%%%%%%%%%%%%%%%%%%%%%%%%%%%%%%%%%%%%%%%%%%%%%%%%%%%%%%%%%%%%%
\section{Connections in the set of indexes. Decompositions}
%%%%%%%%%%%%%%%%%%%%%%%%%%%%%%%%%%%%%%%%%%%%%%%%%%%%%%%%%%%%%
%%%%%%%%%%%%%%%%%%%%%%%%%%%%%%%%%%%%%%%%%%%%%%%%%%%%%%%%%%%%%

At following $\mathbb{W}$ denotes a linear space with a basis $\mathfrak{B}'=\{w_j\}_{j\in J}$ and $\mathbb{V}$ a $k$-module over $\mathbb{W}$ with multiplicative basis $\mathfrak{B} = \{v_i\}_{i\in I}$ with respect to $\mathfrak{B}'$. Firstly we develope connection techniques among the elements in the set of indexes $I$ as the main tool in our study. For each $i \in I$, a new variable $\overline i \notin I$ is introduced and we denote by $$\overline{I} := \{\overline i : i \in I\}$$ the set consisting of all these new symbols. Analogously, for $j \in J$ we define the symbols $\overline{j}$ and the set of symbols $\overline{J}.$ Let $\mathcal P(I)$ be the power set of $I$. From now we denote $\overline{(\overline{x})} := x \in I \cup J$ and $\overline{\mathfrak{A}} := \{\overline i : i \in \mathfrak{A}\}$ for any $\mathfrak{A} \subset \mathcal P(I)$.

Next, we consider the following operation which recover, in a
sense, certain multiplicative relations among the elements of
${\mathfrak B}.$ Let $S_n$ be the group of all permutations of $n$ elements. Given a $\sigma \in S_n$ we define $a_\sigma : I \times \overset{k)}{\cdots} \times I \times J \times \overset{n-k)}{\cdots} \times J \to \mathcal P(I)$ such as
$$\begin{array}{l}
a_\sigma(i_1,\dots,i_k,j_{k+1},\dots,j_n):=\left\{
\begin{array}{cll}
\emptyset &\text{if}&0=[v_{i_1}, \dots, v_{i_k}, w_{j_{k+1}},\dots, w_{j_n}]_{\sigma}\\
\{r_{\sigma}\} &\text{if} & 0 \neq [v_{i_1}, \dots, v_{i_k}, w_{j_{k+1}},\dots, w_{j_n}]_{\sigma} \in\mathbb{F}v_{r_{\sigma}}
\end{array}
\right.\vspace*{0.3cm}\\
\end{array}$$

\noindent and $b_\sigma : I\times \overline{I} \times \overset{k-1)}{\dots} \times \overline{I} \times \overline{J} \times \overset{n-k)}{\dots} \times \overline{J} \to \mathcal{P}(I)$ such as
$$\begin{array}{l}
b_\sigma(i,\overline{i}_2,\dots,\overline{i}_k,\overline{j}_{k+1},\dots,\overline{j}_n) := \Bigl\{i' \in I : a_\sigma(i',i_2,\dots,i_k,j_{k+1},\dots,j_n) = \{i\} \Bigr\}
\end{array}$$

\noindent Then, we consider the following operation $$\mu : I \times (I\; \dot{\cup} \; \overline{I}) \times \overset{k-1)}{\cdots} \times (I \; \dot{\cup} \; \overline{I}) \times (J\;\dot{\cup} \; \overline{J}) \times \overset{n-k)}{\cdots} \times (J \; \dot{\cup} \; \overline{J}) \to \mathcal{P}(I)$$ given by:
\medskip

\begin{itemize}
\item $\displaystyle\mu(i_1,\dots,i_k,j_{k+1},\dots,j_n):=\bigcup_{\sigma\in S_n} a_\sigma(i_1,\dots,i_k,j_{k+1},\dots,j_n)$ for $i_1,\dots,i_k \in I$ and $j_{k+1},\dots,j_n \in J$,
\item $\displaystyle\mu(i,\overline{i}_2,\dots,\overline{i}_k,\overline{j}_{k+1},\dots,\overline{j}_n) :=\bigcup_{\sigma\in S_n} b_\sigma(i,\overline{i}_2,\dots,\overline{i}_k,\overline{j}_{k+1},\dots,\overline{j}_n)$ for $i \in I,$ $\overline{i}_2,\dots,\overline{i}_k \in \overline{I}$ and $\overline{j}_{k+1},\dots,\overline{j}_n \in \overline{J}$.
\end{itemize}

\noindent We define $\mu(i,i_2,\dots,i_k,j_{k+1},\dots,j_n):=\emptyset$ if neither $\{i_2,\dots,i_k\} \subset I$ and $\{j_{k+1},\dots,j_n\} \subset J$ nor $\{i_2,\dots,i_k\} \subset \overline{I}$ and $\{j_{k+1},\dots,j_n\} \subset \overline{J}.$

\begin{remark}\label{remark1}\rm
Let us observe that $$\mu(i_1,\dots,i_k,j_{k+1},\dots,j_n)=\mu(i_{\sigma(1)},\dots,i_{\sigma(k)},j_{k+\theta(1)},\dots,j_{k+\theta(n-k)})$$ for any $\sigma\in S_k$ and $\theta \in S_{n-k}$ if $\{i_1,\dots,i_k\} \in I$ and $\{j_{k+1},\dots,j_n\} \in J$. We have that $\mu$ also verifies $$\mu(i,\overline{i}_2,\dots,\overline{i}_k,\overline{j}_{k+1},\dots,\overline{j}_n)=\mu(i,\overline{i}_{1+\sigma(1)},\dots,\overline{i}_{1+\sigma(k-1)},\overline{j}_{k+\theta(1)},\dots,\overline{j}_{k + \theta(n-k)})$$ for any $\sigma\in S_{k-1}$ and $\theta \in S_{n-k}$ if $\{\overline{i}_2,\dots,\overline{i}_k\} \in \overline{I}$ and $\{\overline{j}_{k+1},\dots,\overline{j}_n\} \in \overline{J}$.
\end{remark}

Now, we also consider the mapping $$\phi: \mathcal{P}(I) \times (I\;\dot\cup\;\overline{I}) \times \overset{k-1)}{\cdots} \times (I\;\dot\cup\;\overline{I}) \times (J\;\dot\cup\;\overline{J}) \times \overset{n-k)}{\cdots} \times (J\;\dot\cup\;\overline{J}) \to \mathcal{P}(I)$$ defined as $$\phi(\mathfrak{A},x_2,\dots,x_k,y_{k+1},\dots,y_n):=\bigcup_{i\in \mathfrak{A}} \mu(i,x_2,\dots,x_k,y_{k+1},\dots,y_n),$$ where  $\mathfrak{A} \in\mathcal{P}(I),$ $\{x_2,\dots,x_k\} \in I\;\dot\cup\;\overline{I}$ and $\{y_{k+1},\dots,y_n\} \in J\;\dot\cup\;\overline{J}.$ From Remark \ref{remark1} we have for any $\sigma \in S_{k-1}$ and $\theta \in S_{n-k},$ $$\phi(\mathfrak{A}, x_2, \dots, x_k, y_{k+1},\dots,y_n) = \phi(\mathfrak{A}, \overline{x}_{1+\sigma(1)},\dots,\overline{x}_{1+\sigma(k-1)},\overline{y}_{k+\theta(1)},\dots,\overline{y}_{k+\theta(n-k)}).$$

\noindent Next result show us the relation by $\mu$ among elements in $I\;\dot\cup\;\overline{I}$.

\begin{lemma}\label{lema1}
Let $i,i' \in I$. Given $\{x_2,\dots,x_k\} \in I\;\dot\cup\;\overline{I}$ and $\{y_{k+1},\dots,y_n\} \in J\;\dot\cup\;\overline{J}$ then $i \in \mu(i',x_2,\dots,x_k,y_{k+1},\dots,y_n)$ if and only if $i' \in \mu(i,\overline{x}_2,\dots,\overline{x}_k,\overline{y}_{k+1},\dots,\overline{y}_n)$.
\end{lemma}

\begin{proof}
To prove the first implication let us suppose that $i \in \mu(i',x_2,\dots,x_k,y_{k+1},\dots,y_n).$ In the case $\{x_2,\dots,x_k\} \in I$ and $\{y_{k+1},\dots,y_n\} \in J$, there exists $\sigma \in S_n$ such that $$\{i\} = a_\sigma(i',x_2,\dots,x_k,y_{k+1},\dots,y_n),$$ then $i' \in b_\sigma(i,\overline{x}_2,\dots,\overline{x}_k,\overline{y}_{k+1},\dots,\overline{y}_n) \subset \mu(i,\overline{x}_2,\dots,\overline{x}_k,\overline{y}_{k+1},\dots,\overline{y}_n)$. In the another case, if $\{x_2,\dots,x_k\} \in \overline{I}$ and $\{y_{k+1},\dots,y_n\} \in \overline{J}$ then exists $\sigma \in S_n$ satisfying $i \in b_\sigma(i',x_2,\dots,x_k,y_{k+1},\dots,y_n)$ and so $$\{i'\} = a_\sigma(i,\overline{x}_2,\dots,\overline{x}_k,\overline{y}_{k+1},\dots,\overline{y}_n) \subset \mu(i,\overline{x}_2,\dots,\overline{x}_k,\overline{y}_{k+1},\dots,\overline{y}_n).$$
To prove the converse we can argue in a similar way.
\end{proof}

\noindent As consequence of Lemma \ref{lema1} we can state the next result.

\begin{lemma}\label{lema2}
Let $\{x_2,\dots,x_k\} \in I\;\dot\cup\;\overline{I},$ $\{y_{k+1},\dots,y_n\} \in J\;\dot\cup\;\overline{J}$ and $\mathfrak{A} \in\mathcal{P}(I).$ It holds that $i \in \phi(\mathfrak{A},x_2,\dots,x_k,y_{k+1},\dots,y_n)$ if and only if $\emptyset \neq \phi(\{i\},\overline{x}_2,\dots,\overline{x}_k,\overline{y}_{k+1},\dots,\overline{y}_n) \cap \mathfrak{A}$.
\end{lemma}

\begin{proof}
Let us suppose that $i \in \phi(\mathfrak{A}, x_2, \dots, x_k, y_{k+1}, \dots, y_n)$. Then there exists $i' \in \mathfrak{A}$ such that $i \in \mu(i',x_2,\dots,x_k,y_{k+1},\dots,y_n)$. By Lemma \ref{lema1} $$i' \in\mu(i,\overline{x}_2,\dots,\overline{x}_k,\overline{y}_{k+1},\dots,\overline{y}_n) \subset \phi(\{i\},\overline{x}_2,\dots,\overline{x}_k, \overline{y}_{k+1},\dots,\overline{y}_n).$$ So
$$i' \in\phi(\{i\},\overline{x}_2,\dots,\overline{x}_k, \overline{y}_{k+1},\dots,\overline{y}_n) \cap \mathfrak{A} \neq \emptyset.$$

\noindent Arguing in a similar way the converse is proven.
\end{proof}

For an easier comprenhesion we firstly present a shorter notation. Let $m$ be a natural number, we denote $X_m := (x_{m,2},\dots,x_{m,k}) \in I\;\dot\cup\; \overline{I} \times \overset{k-1)}{\dots} \times I\;\dot\cup\; \overline{I}$ and $Y_m := (y_{m,k+1},\dots,y_{m,n}) \in J\;\dot\cup\; \overline{J} \times \overset{n-k)}{\dots} \times J\;\dot\cup\; \overline{J}$. Let us also denote $$\overline{X}_m := (\overline{x}_{m,2},\dots,\overline{x}_{m,k}) \hspace{0.3cm} \mbox{and} \hspace{0.3cm} \overline{Y}_m := (\overline{y}_{m,k+1},\dots,\overline{y}_{m,n}).$$

\noindent Additionally, for $t \geq 1,$ by $\{X_1,Y_1,\dots,X_t,Y_t\}$ we mean the set of elements $$\{x_{1,2},\dots,x_{1,k},y_{1,k+1},\dots,y_{1,n}, \dots, x_{t,2},\dots,x_{t,k},y_{t,k+1},\dots,y_{t,n}\}.$$

\noindent Finally, for $\mathfrak{A} \in \mathcal{P}(I)$ we denote $\phi(\mathfrak{A},X_m,Y_m) := \phi(\mathfrak{A},x_{m,2},\dots,x_{m,k},y_{m,k+1},\dots,y_{m,n}).$

\begin{definition}\label{connection}\rm
Let $i$ and $i'$ be distinct elements in $I$. We say that $i$ is {\it connected} to $i'$ if there exists a subset $\{X_1,Y_1,\dots,X_t,Y_t\} \subset I\;\dot\cup\; \overline{I} \; \dot \cup\;  J \; \dot \cup\; \overline{J},$ for certain $t \geq 1$, such that the following conditions hold:

\begin{enumerate}
\item [{\rm 1.}] $\phi(\{i\},X_1,Y_1) \neq\emptyset$,\\
$\phi(\phi(\{i\},X_1,Y_1),X_2,Y_2) \neq\emptyset$,\\
$\hspace*{2cm} \vdots$\\
$\phi(\phi(\dots\phi(\{i\},X_1,Y_1),\dots), X_{t-1},Y_{t-1}) \neq\emptyset$.

\bigskip

\item [{\rm 2.}] $i' \in \phi(\phi(\dots\phi(\{i\},X_1,Y_1),\dots),X_t,Y_t).$

%\item [{\rm 1.}] $\phi(\{i\},x_{1,2},\dots,x_{1,k},y_{1,k+1},\dots, y_{1,n}) \neq\emptyset$,\\
%$\phi(\phi(\{i\},x_{1,2},\dots,x_{1,k},y_{1,k+1},\dots, y_{1,n}),x_{2,2},\dots,x_{2,k},y_{2,k+1},\dots, y_{2,n}) \neq\emptyset$,\\
%$\hspace*{2cm} \vdots$\\
%$\small{\phi(\phi(\dots\phi(\{i\},x_{1,2},\dots,x_{1,k},y_{1,k+1}\dots, y_{1,n}) \dots), x_{T-1,2},\dots,x_{T-1,k},y_{T-1,k+1},\dots, y_{T-1,n}) \neq\emptyset}$.

%\bigskip

%\item [{\rm 2.}] $i' \in \phi(\phi(\dots\phi(\{i\},x_{1,2},\dots,x_{1,k},y_{1,k+1},\dots, y_{1,n})\dots),x_{T,2},\dots,x_{T,k},y_{T,k+1},\dots, y_{T,n}).$
\end{enumerate}

\bigskip

\noindent The subset $\{X_1,Y_1,\dots,X_t,Y_t\}$ is a {\it connection} from $i$ to $i'$ and we accept $i$ to be connected to itself.
\end{definition}

Our aim is to show that the connection relation is of
equivalence. Previously we check the symmetric property.

\begin{lemma}\label{lema3}
Let $\{X_1,Y_1,\dots,X_t,Y_t\}$ be a connection from some $i$ to some $i'$ where $i,i' \in I$ with $i \neq i'$ and $t \geq 1$. Then the set $\{\overline{X}_t,\overline{Y}_t, \dots,\overline{X}_1,\overline{Y}_1\}$ is a connection from $i'$ to $i$.
\end{lemma}

\begin{proof}
Let us prove by induction on $t$.

For $t=1$ we have that $i' \in \phi(\{i\},X_1,Y_1)$. It means that $i' \in \mu(i,X_1,Y_1)$ so, by Lemma \ref{lema1}, $i \in\mu(i',\overline{X}_1,\overline{Y}_1) \subset \phi(\{i'\},\overline{X}_1,\overline{Y}_1)$. Hence $\{\overline{X}_1,\overline{Y}_1\}$ is a connection from $i'$ to $i$.

Let us suppose that the assertion holds for any connection with certain $t$, where $t \geq 1$, and let us show this assertion also holds for any connection from $i$ to $i'$ of the form $$\{X_1,Y_1,\dots,X_t,Y_t,X_{t+1},Y_{t+1}\}.$$

\noindent Denoting the set $\mathfrak{A} := \phi(\phi(\dots\phi(\{i\},X_1,Y_1),\dots),X_t,Y_t)$ taking into the account second condition of the Definition \ref{connection} we have that
$$i' \in\phi(\mathfrak{A}, X_{t+1},Y_{t+1}).$$

\noindent Then, by the Lemma \ref{lema2}, $\phi(\{i'\},\overline X_{t+1},\overline Y_{t+1}) \cap \mathfrak{A} \neq \emptyset$ so we can take $h \in \mathfrak{A}$ such that
\begin{equation}\label{eqq1}
h \in \phi(\{i'\},\overline X_{t+1},\overline Y_{t+1}).
\end{equation}

\noindent From $h \in \mathfrak{A}$ we have that $\{X_1,Y_1,\dots, X_t, Y_t\}$ is a connection from $i$ to $h$, so by induction hypothesys we get $$\{\overline{X}_t,\overline{Y}_t, \dots,\overline{X}_1,\overline{Y}_1\}$$ connecting $h$ with $i$. From here and Equation \eqref{eqq1} we have
$$i \in \phi(\phi(\dots\phi(\phi(\{i'\},\overline X_{t+1},\overline Y_{t+1}),\overline{X}_t,\overline{Y}_t)\dots),\overline{X}_1,\overline{Y}_1).$$

\noindent So $\{\overline X_{t+1},\overline X_{t+1}, \overline X_t, \overline Y_t,\dots,\overline X_1, \overline Y_1\}$ is a connection from $i'$ to $i$, which completes the proof.
\end{proof}

\noindent Now we can assert next result.

\begin{proposition}
The relation $\sim$ in $I$, defined by $i \sim i'$ if and only if $i$ is connected to $i'$, is an equivalence relation.
\end{proposition}

\begin{proof}
The reflexive and the symmetric character of $\sim$ is given by  Definition \ref{connection} and Lemma \ref{lema3}. If we consider the connections $\{X_1,Y_1,\dots,X_t,Y_t\}$ and $\{X'_1,Y'_1,\dots,X'_p,Y'_p\}$ from $i$ to $i'$ and from $i'$ to $i''$ respectively, then it is easy to prove that $$\{X_1,Y_1, \dots,X_t,Y_t,X'_1,Y'_1,\dots,X'_p,Y'_p\}$$ is a connection from $i$ to $i''$. So $\sim$ is transitive and consequently an equivalence relation.
\end{proof}

By the above proposition we can introduce the quotient set
$$I/\sim := \{[i] : i \in I\},$$ where $[i]$ denotes the class of equivalence of $i$, that is, the set of elements in $I$ which are connected to $i$. For any $[i] \in I/\sim$ we define the linear subspace $$V_{[i]} := \bigoplus_{i'\in [i]}\mathbb{F}v_{i'}.$$

\noindent A linear subspace $U$ of a $k$-module $\mathbb{V}$ is said a $k${\it -submodule} if it satisfies $$[U, \mathbb{V}, \overset{k-1)}{\dots},\mathbb{V},\mathbb{W},\overset{n-k)}{\dots},\mathbb{W}]_{\sigma} \subset U$$ for any $\sigma \in S_n$.

\begin{proposition}\label{lema_submodulo}
For any $[i] \in I/\sim$ we have that $V_{[i]}$ is a $k$-submodule of $\mathbb{V}$.
\end{proposition}
\begin{proof}
We need to check $[V_{[i]},\mathbb{V},\overset{k-1)}{\dots},\mathbb{V},\mathbb{W},\overset{n-k)}{\dots},\mathbb{W}]_{\sigma} \subset V_{[i]}$ for any $\sigma \in S_n$. Suppose there exist $i_1 \in [i], i_2,\dots,i_k \in I$ and $j_{k+1},\dots, j_n \in J$ such that $$0 \neq [v_{i_1},v_{i_2},\dots, v_{i_k}, w_{j_1},\dots,w_{j_n}]_{\sigma} \in \mathbb{F}v_r,$$ for some $r \in I$. Therefore $r \in \phi(\{i_1\},i_2,\dots,i_k,j_{k+1},\dots,j_n)$. Considering the connection $\{i_2,\dots,i_k,j_{k+1},\dots,j_n\}$ we get $i_1 \sim r$, and by transitivity $r \in [i]$. Consequently $0 \neq [v_{i_1},\dots, v_{i_k},w_{j_1},\dots,w_{j_n}]_{\sigma} \in V_{[i]}$ and so $V_{[i]}$ is a $k$-submodule of $\mathbb{V}$.
\end{proof}

\noindent By Proposition \ref{lema_submodulo} we have next result.

\begin{proposition}\label{lema_producto}
For any $[i],[h] \in I/\sim$ such that $[i] \neq [h]$ we have $$[V_{[i]},V_{[h]},\mathbb{V},\overset{k-2)}{\dots},\mathbb{V},\mathbb{W},\overset{n-k)}{\dots},\mathbb{W}]_{\sigma} = 0.$$
\end{proposition}

\begin{definition}\label{inherited_basis}\rm
We say that a $k$-submodule $U$ of $\mathbb{V}$ admits a multiplicative basis $\mathfrak{B}_U$ {\it inherited} from $\mathfrak{B}$ if $\mathfrak{B}_U \subset \mathfrak{B}$.
\end{definition}

\noindent Observe that any $k$-submodule $V_{[i]} \subset \mathbb{V}$ admits an inherited basis $\mathfrak{B}_{[i]} := \{v_{[i']} : i' \in [i]\}$. So we can assert
%$\mathfrak{B}_{[i]}=\{w_j :j\in[i]\}$ is a basis of $T_{[i]}$ for any $[i]\in I/\sim$.
%Given $x,y,z\in[i]$ such that $0\neq\langle w_x,w_y,w_z\rangle\in \mathbb{F}w_w$ for some $w\in I$ we have that $w\in\phi(\{x\},y,z)$. Then $\{x,y,z\}$ is a connection from $x$ to $w$, and by transitivity $w\in[i]$. So $w_w\in
%\mathfrak{B}_{[i]}$ concluding that $\mathfrak{B}_{[i]}$ is a inherited from $\mathfrak{B}$.

\begin{theorem}\label{theo1}
Let $\mathbb{V}$ be a $k$-module admitting a multiplicative basis $\mathfrak{B}$ with respect to a fixed basis of $\mathbb{W}$. Then $$\mathbb{V} = \bigoplus_{[i]\in I/\sim}V_{[i]},$$ being any $V_{[i]} \subset \mathbb{V}$ a $k$-submodule admitting a multiplicative basis $\mathfrak{B}_{[i]}$ inherited from $\mathfrak{B}$. Moreover, it is satisfied $[V_{[i]},V_{[h]},\mathbb{V}, \overset{k-2)}{\dots},\mathbb{V},\mathbb{W},\overset{n-k)}{\dots},\mathbb{W}]_{\sigma} = 0,$ with $\sigma \in S_n,$ in case $[i] \neq [h].$
\end{theorem}

\noindent We call that a $k$-module $\mathbb{V}$ is {\it simple} if its only $k$-submodules are $\{0\}$ and $\mathbb{V}$.

\begin{corollary}\label{coro1}
If $\mathbb{V}$ is simple then any couple of elements of $I$ are connected.
\end{corollary}

\begin{proof}
The simplicity of $\mathbb{V}$ applies to get that $V_{[i]} = \mathbb{V}$ for some $[i] \in I/\sim$. Hence $[i] = I$ and so any couple of elements in $I$ are connected.
\end{proof}

%%%%%%%%%%%%%%%%%%%%%%%%%%%%%%%%%%%%%%%%%%%%%%%%%%%%%%%%%%%%%
%%%%%%%%%%%%%%%%%%%%%%%%%%%%%%%%%%%%%%%%%%%%%%%%%%%%%%%%%%%%%
\section{The minimal components}
%%%%%%%%%%%%%%%%%%%%%%%%%%%%%%%%%%%%%%%%%%%%%%%%%%%%%%%%%%%%%
%%%%%%%%%%%%%%%%%%%%%%%%%%%%%%%%%%%%%%%%%%%%%%%%%%%%%%%%%%%%%

In this section we show that, under mild conditions, the decomposition of $\mathbb{V}$ presented in Theorem \ref{theo1} can be given by means of the family of its minimal $k$-submodules. We begin by introducing a concept of minimality for $k$-modules that agree with the one for algebras in \cite{Yo_Arb_Alg} and for modules in \cite{Yo_modules}.

\begin{definition}\rm
A $k$-module $\mathbb{V}$ over a linear space $\mathbb{W}$ admitting a multiplicative basis $\mathfrak{B}$ with respect to a fixed basis of $\mathbb{W}$ is said to be {\it minimal} if its unique non-zero $k$-submodule admitting a multiplicative basis inherited from $\mathfrak{B}$ is $\mathbb{V}$.
\end{definition}

Let us also introduce the concept of $\mu$-multiplicativity in the framework of $k$-modules over linear spaces in a similar way to the analogous one for arbitrary algebras and modules over linear spaces (see \cite{Yo_modules,Yo_Arb_Alg} for these notions and examples).

\begin{definition}\rm
We say that a $k$-module $\mathbb{V}$ over a linear space $\mathbb{W}$ admits a $\mu${\it -multiplicative basis} $\mathfrak{B} = \{v_i\}_{i\in I}$ with respect to a fixed basis $\mathfrak{B}'=\{w_j\}_{j \in J}$ of $\mathbb{W}$, if it is multiplicative and given $i,i' \in I$ such that $i' \in \mu(i,i_2,\dots,i_k,j_{k+1},\dots,j_n)$ for some $i_2,\dots,i_k \in I \;\dot\cup\;\overline{I}$ and $j_{k+1},\dots,j_n \in J \;\dot\cup\;\overline{J}$ then $v_{i'} \in [v_i,\mathbb{V},\overset{k-1)}{\dots},\mathbb{V},\mathbb{W},\overset{n-k)}{\dots},\mathbb{W}]_{\sigma}$.
\end{definition}

\begin{theorem}\label{theo2}
Let $\mathbb{V}$ be a $k$-module over a linear space $\mathbb{W}$ admitting a $\mu$-multiplicative basis $\mathfrak{B} = \{v_i\}_{i \in I}$ with respect to the basis $\mathfrak{B}' = \{w_j\}_{j \in J}$ of $\mathbb{W}$. It holds that $\mathbb{V}$ is minimal if and only if the set of indexes $I$ has all of its elements connected.
\end{theorem}

\begin{proof}
Suppose $\mathbb{V}$ is minimal. Taking into account the decomposition of $\mathbb{V}$ given in Theorem \ref{theo1}, $\mathbb{V} = \bigoplus_{[i] \in I/\sim}V_{[i]},$ and that any $V_{[i]}$ admits a multiplicative basis, we get $V_{[i]} = \mathbb{V}$ for some $[i] \in I/\sim$ and so $[i] = I$.

To prove the converse, consider a non-zero $k$-submodule $U$ of $\mathbb{V}$ admitting a multiplicative basis inherited from $\mathfrak B$. Then, for a certain $\emptyset \neq I_U \subset I$, we can write $U = \bigoplus_{i \in I_U}\mathbb Fv_i$. Fix some $i_0 \in I_U$ being then
\begin{equation}\label{eqq5}
0 \neq v_{i_0}\in U.
\end{equation}
We show by induction on $t$ that if $\{X_1,Y_1,\dots,X_t,Y_t\}$ is any connection from $i_0$ to some $i' \in I$ then for any $h \in \phi(\phi(\cdots\phi(\{i_0\},X_1,Y_1)\dots),X_t,Y_t)$ we have that $0 \neq v_h \in U$.

In case $t=1$, we get $h \in\phi(\{i_0\},X_1,Y_1)$. Hence $h \in \mu(i_0,i_2,\dots,i_k,j_{k+1},\dots,j_n)$, then, taking into account that $U$ is a $k$-submodule of $\mathbb{V}$, the $\mu$-multiplicativity of $\mathfrak B$ and Equation \eqref{eqq5} we obtain $v_h \in [v_{i_0},\mathbb{V},\overset{k-1)}{\dots},\mathbb{V},\mathbb{W},\overset{n-k)}{\dots},\mathbb{W}]_{\sigma} \subset U$.

Suppose now that the assertion holds for any connection $\{X_1,Y_1,\dots,X_t,Y_t\}$ from $i_0$ to some $r \in I.$ Consider an arbitrary connection $\{X'_1,Y'_1,\dots,X'_t,Y'_t,X'_{t+1},Y'_{t+1}\}$ from $i_0$ to some $i' \in I$. By induction hypothesis we know that for $g \in \mathfrak{A}$, where $\mathfrak{A} := \phi(\phi(\cdots \phi(\{i_0\},X'_1,Y'_1)\cdots), X'_t,Y'_t)$, the element
\begin{equation}\label{eqq6}
0 \neq v_g \in U.
\end{equation}
Taking into account that the fact $h \in \phi(\phi(\cdots \phi(\{i_0\},X'_1,Y'_1)\dots), X'_{t+1},Y'_{t+1})$ means $h \in \phi(\mathfrak{A}, X'_{t+1},Y'_{t+1}),$ we have that $h \in \mu(g,X_{t+1},Y_{t+1}),$ with $g \in \mathfrak{A}$. From here, the $\mu$-multiplicativity of $\mathfrak B$ and Equation \eqref{eqq6} imply that $v_h \in [v_g, \mathbb{V}, \overset{k-1)}{\dots},\mathbb{V},\mathbb{W},\overset{n-k)}{\dots},\mathbb{W}]_{\sigma} \subset U$ as desired.

Given any $i' \in I$ we know that $i_0$ is connected to $i'$, so we can assert by the above observation that $\mathbb Fv_{i'} \subset U$. We have shown $\mathbb{V} = \bigoplus_{i' \in I}\mathbb Fv_{i'} \subset U$ and so $U = \mathbb{V}$.
\end{proof}

\begin{theorem}
Let $\mathbb{V}$ be a $k$-module over a linear space $\mathbb{W}$ admitting a $\mu$-multiplicative basis $\mathfrak{B}$ with respect to a fixed basis of $\mathbb{W}$. Then $\mathbb{V} = \bigoplus_{\alpha \in \Omega} V_{\alpha}$ is the direct sum of the family of its minimal $k$-submodules, each one admitting a $\mu$-multiplicative basis inherited from $\mathfrak B$ and in such a way that $[V_{\alpha},V_{\gamma},\mathbb{V}, \overset{k-2)}{\dots},\mathbb{V},\mathbb{W},\overset{n-k)}{\dots},\mathbb{W}]_{\sigma} = 0,$ with $\sigma \in S_n,$ in case $\alpha \neq \gamma.$
\end{theorem}

\begin{proof}
By Theorem \ref{theo1}, $\mathbb{V}$ is the direct sum of the $k$-submodules $V_{[i]}$ (with $[i] \in I/\sim$) satisfying $[V_{[i]},V_{[h]},\mathbb{V},\overset{k-2)}{\dots},\mathbb{V},\mathbb{W},\overset{n-k)}{\dots},\mathbb{W}]_{\sigma} = 0,$ with $\sigma \in S_n,$ if $[i] \neq [h].$
Now for any $V_{[i]}$ we have that $\mathfrak B_{[i]}$ is a $\mu$-multiplicative basis where all of the elements of $[i]$ are connected. Applying Theorem \ref{theo2} to any $V_{[i]}$ we have that the decomposition $\mathbb{V} =\bigoplus_{[i]\in I/\sim} V_{[i]}$ satisfies the assertions of the theorem.
\end{proof}

%%%%%%%%%%%%%%%%%%%%%%%%%%%%%%%%%%%%%%%%%%%%%%%%%%%%%%%%%%%%%
%%%%%%%%%%%%%%%%%%%%%%%%%%%%%%%%%%%%%%%%%%%%%%%%%%%%%%%%%%%%%
\section{$k$-modules over arbitrary $n$-ary algebras with multiplicative basis}
%%%%%%%%%%%%%%%%%%%%%%%%%%%%%%%%%%%%%%%%%%%%%%%%%%%%%%%%%%%%%
%%%%%%%%%%%%%%%%%%%%%%%%%%%%%%%%%%%%%%%%%%%%%%%%%%%%%%%%%%%%%

In this section we study an application of $k$-modules with a multiplicative basis over an arbitrary $n$-ary algebra with multiplicative basis. We consider a $n$-ary algebra $\mathbb{A}$ of arbitrary dimension and over an arbitrary base field $\mathbb F$ with a multiplicative basis $\mathfrak{B}' = \{e_j\}_{j \in J}.$ That is, for any $j_1, \dots, j_n \in J$ we have $\langle e_{j_1},\dots,e_{j_n} \rangle \in\mathbb{F}e_j$ for some $j \in J$.

\noindent Let $\mathbb{V}$ be a $k$-module over an arbitrary $n$-ary algebra $\mathbb{A}$. Assume that $\mathbb{V}$ has a multiplicative basis $\mathfrak{B} = \{v_i\}_{i \in I}$ with respect to the multiplicative basis $\mathfrak{B}'$. We say $\mathbb{V}$ is a $k$-module over $\mathbb{A}$ under the action $[\mathbb{V},\overset{k)}{\dots},\mathbb{V},\mathbb{A},\overset{n-k)}{\dots},\mathbb{A}]_{\sigma} \subset \mathbb{V}$, for any permutation $\sigma$ of $n$ elements.

Consider the $\mathbb{F}$-linear space $\mathbb{A} \oplus \mathbb{V}$ with basis $\mathfrak{B}'' := \mathfrak{B} \dot{\cup} \mathfrak{B}'$ and define the structure of $k$-module over itself determined by
\begin{eqnarray}
\begin{split}\label{action}
& (\mathbb{A} \oplus \mathbb{V}) \times \overset{n)}{\dots} \times (\mathbb{A} \oplus \mathbb{V}) \longrightarrow (\mathbb{A} \oplus \mathbb{V}) \\
& \llbracket x_1,\dots,x_n \rrbracket := \langle x_1,\dots,x_n \rangle \in \mathbb{A} \\
& \llbracket y_1,\dots,y_k,x_{k+1},\dots,x_n \rrbracket_{\sigma} := [y_1,\dots,y_k,x_{k+1},\dots,x_n]_{\sigma} \in \mathbb{V} \\
& \llbracket y_1,\dots,y_t,x_{t+1},\dots,x_n \rrbracket_{\sigma} := 0, \hspace{0.2cm} 1 \leq t \leq n, t \neq k
\end{split}
\end{eqnarray}
for $x_1,\dots,x_n \in \mathbb{A}, y_1,\dots,y_n \in \mathbb{V}$ and $\sigma$ any permutation of $n$ elements. Clearly, $\mathbb{A} \oplus \mathbb{V}$ is a $k$-module with respect to itself with a multiplicative basis $\mathfrak{B}''$. So Theorem \ref{theo1} let us assert that
\begin{equation}\label{application}
\mathbb{A} \oplus \mathbb{V} = \bigoplus_{\alpha \in \Omega} U_{\alpha},
\end{equation}
being each $U_{\alpha}$ a $k$-submodule of $\mathbb{A} \oplus \mathbb{V}.$ Observe that the decomposition \eqref{application} and the construction of any $U_{\alpha}$ imply $\mathbb{V} = \oplus_{\alpha \in \Omega} (U_{\alpha} \cap \mathbb{V})$ and $\mathbb{A} = \oplus_{\alpha \in \Omega} (U_{\alpha} \cap \mathbb{A}).$ By denoting $V_{\alpha} := U_{\alpha} \cap \mathbb{V}, A_{\alpha} := U_{\alpha} \cap \mathbb{A}$ for any $\alpha \in \Omega,$ and $\Omega_{\mathbb{V}} := \{\alpha \in \Omega : V_{\alpha} \neq 0\},$ $\Omega_{\mathbb{A}} := \{\alpha \in \Omega : A_{\alpha} \neq 0\}$ we can write $$\mathbb{V} = \bigoplus_{\alpha \in \Omega_{\mathbb{V}}} V_{\alpha}, \hspace{0.2cm} \mbox{and} \hspace{0.2cm} \mathbb{A} = \bigoplus_{\alpha \in \Omega_{\mathbb{A}}} A_{\alpha}.$$ As consequence of the action \eqref{action}, any $V_{\alpha}$ is a $k$-module with respect to $\mathbb{A}$ and by \cite[Lemma 2.4]{Yo_n_algebras} any $A_{\beta}$ is an ideal of the $n$-ary algebra $\mathbb{A}$.

\begin{lemma}\label{lema_ultimo}
For any $\alpha \in \Omega_V$ such that $\llbracket V_{\alpha}, \mathbb{V}, \overset{k-1)}{\dots},\mathbb{V}, \mathbb{A}, \overset{n-k)}{\dots}, \mathbb{A} \rrbracket_{\sigma} \neq 0$, there exists a unique $\beta \in \Omega_{\mathbb{A}}$ such that $\llbracket V_{\alpha}, \mathbb{V}, \overset{k-1)}{\dots},\mathbb{V}, A_{\beta}, \mathbb{A}, \overset{n-k-1)}{\dots}, \mathbb{A} \rrbracket_{\sigma} \neq 0$.
\end{lemma}

\begin{proof}
In the opposite case, there exist $v_p, v_s \in \mathfrak{B} \cap V_{\alpha}$, $e_i \in \mathfrak{B}' \cap A_{\beta}$ and $e_j \in \mathfrak{B}' \cap A_{\gamma}$ with $\beta, \gamma \in \Omega_{\mathbb{A}}, \beta \neq \gamma$, such that $$0 \neq \llbracket v_p,v_{i_2},\dots,v_{i_k},e_i,e_{j_{k+2}},\dots,e_{j_n} \rrbracket_{\sigma} \in \mathbb{F}v_h$$ and $$0 \neq \llbracket v_s,v_{l_2},\dots,v_{l_k},e_j,e_{q_{k+2}},\dots,e_{q_n} \rrbracket_{\sigma} \in \mathbb{F}v_{h'},$$ with $v_{i_2},\dots,v_{i_k},v_{l_2},\dots,v_{l_k} \in \mathfrak{B},$ $e_{j_{k+2}},\dots,e_{j_n},e_{q_{k+2}},\dots, e_{q_n} \in \mathfrak{B}'$ and certains $v_h,v_{h'} \in \mathbb{F}$. Since $[p] = [s] = \alpha$ (see Theorem \ref{theo1}) there exists a connection $\{X_1,Y_1,\dots,X_t,Y_t\}$ from $p$ to $s$ in case $p \neq s$. Let us define $P := (i_2,\dots,i_k,j_{k+2},\dots,j_n)$ and $Q := (l_2,\dots,l_k,,q_{k+2},\dots,q_n).$ Let us fix the notation for $\mathcal{A} := (a_1,\dots,a_{n-2}),$  $(a, \mathcal{A}) := (a,a_1,\dots,a_{n-2})$ and $\phi(\{b\},(a, \mathcal{A})) := \phi(\{b\},a,a_1,\dots,a_{n-2}),$ with $a, a_1,\dots,a_r \in I \cup J \cup \overline{I} \cup \overline{J}$ and $b \in I$. Then we have that either the set $\{(p, P), (\overline{i}, \overline{P}), (j, Q), (\overline{s}, \overline{Q})\}$ (if $p = s$) is a connection from $i$ to $j$ because
\begin{eqnarray*}
\begin{split}
& h \in \phi(\{i\},(p, P)) \neq \emptyset\\
& s = p \in \phi(\phi(\{i\},(p, P)),(\overline{i}, \overline{P})) \neq \emptyset \\
& h' \in \phi(\phi(\phi(\{i\}, (p, P)), (\overline{i}, \overline{P})),(j, Q)) \neq \emptyset \\
& j \in \phi(\phi(\phi(\phi(\{i\},(p, P)), (\overline{i},\overline{P})), (j, Q)), (\overline{s}, \overline{Q})) \\
\end{split}
\end{eqnarray*}
or in a similar way we can use the set $\{(p, P), (\overline{i}, \overline{P}), X_1,Y_1,\dots,X_t,Y_t,(j, Q), (\overline{s}, \overline{Q})\}$ (if $p \neq s$) to conclude again that $[i] = [j]$. That is (we apply Theorem \ref{theo1} on $A$ because it can be clearly seen as an $n$-module over itself), $\beta = \gamma$, a contradiction. In a similar way, we can show that for any $\beta \in \Omega_{\mathbb{A}}$ such that $\llbracket \mathbb{V}, \overset{k)}{\dots},\mathbb{V}, A_{\beta}, \mathbb{A}, \overset{n-k-1)}{\dots},\mathbb{A} \rrbracket_{\sigma} \neq 0$, there exists a unique $\alpha \in \Omega_{\mathbb{V}}$ such that $\llbracket V_{\alpha}, \mathbb{V}, \overset{k-1)}{\dots},\mathbb{V}, A_{\beta}, \mathbb{A}, \overset{n-k-1)}{\dots},\mathbb{A} \rrbracket_{\sigma} \neq 0$.
\end{proof}

By applying Theorems \ref{theo1}, \ref{theo2} and Lemma \ref{lema_ultimo} we can assert the next result.

\begin{corollary}
Let $\mathbb{A}$ be an arbitrary $n$-ary algebra with a multiplicative basis $\mathfrak{B}'$ and $\mathbb{V}$ a $k$-module over the $n$-ary algebra $\mathbb{A}$ admitting a multiplicative basis $\mathfrak{B}$ with respect to $\mathfrak{B}'$. Then $$\mathbb{V} = \bigoplus_{\alpha \in \Omega_V}V_{\alpha},$$ being any $V_{\alpha}$ a $k$-submodule of $\mathbb{V}$ admitting a multiplicative basis inherited from the one of $\mathbb{V}$, and $$\mathbb{A} = \bigoplus_{\beta \in \Omega_A}A_{\beta},$$ being any $A_{\beta}$ an ideal of $\mathbb{A}$ admitting a multiplicative basis inherited from the one of $\mathbb{A}$. Furthermore, if we denote by $\Omega'_{\mathbb{V}} := \{\alpha \in \Omega_{\mathbb{V}} : \llbracket V_{\alpha},\mathbb{V},\overset{k-1)}{\dots},\mathbb{V},\mathbb{A},\overset{n-k)}{\dots},\mathbb{A} \rrbracket_{\sigma} \neq 0\}$ and by $\Omega'_{\mathbb{A}} := \{\beta \in \Omega_{\mathbb{A}} : \llbracket \mathbb{V},\overset{k)}{\dots},\mathbb{V},A_{\beta},\mathbb{A},\overset{n-k-1)}{\dots},\mathbb{A} \rrbracket_{\sigma} \neq 0\},$ there exists a bijection $f : \Omega'_{\mathbb{V}} \to \Omega'_{\mathbb{A}}$ satisfying that $\llbracket V_{\alpha},\mathbb{V},\overset{k-1)}{\dots},\mathbb{V},A_{f(\alpha)},\mathbb{A},\overset{n-k-1)}{\dots},\mathbb{A} \rrbracket_{\sigma} \neq 0$ and $\llbracket V_{\alpha},\mathbb{V},\overset{k-1)}{\dots},\mathbb{V},A_{\beta},\mathbb{A},\overset{n-k-1)}{\dots},\mathbb{A} \rrbracket_{\sigma} = 0$ if $\beta \neq f(\alpha)$.
If $\mathfrak{B}''$ is also $\mu$-multiplicative, then the above decompositions of $\mathbb{V}$ and $\mathbb{A}$ are by means of the family of its minimal $k$-submodules and of the family of its minimal ideals, respectively.
\end{corollary}

\end{document}